\documentclass[12pt]{article}
\usepackage{amsmath}
\usepackage[dvipsnames,usenames]{color}
\usepackage{amsfonts}
\usepackage{mathrsfs}
\usepackage{amssymb}
\usepackage{amsthm}
\usepackage{mathrsfs} %cctart为文章，cctbook为书籍
\usepackage{amsmath, amsfonts, latexsym, amsthm, amsxtra, amssymb} %数学包
\usepackage{graphics}
\def\endpf{\hbox{\vrule height1.5ex width.5em}}

\def\<{\langle}
\def\>{\rangle}
\numberwithin{equation}{section}
\def\<{\langle}
\def\>{\rangle}

\def\ra{\rightarrow}
\def\p{\partial}
\def\a{\alpha}

\def\wt{\widetilde}

\def \sm{\setminus}
\def\-{\overline}

\def\e{\epsilon}

\def\endpf{\hbox{\vrule height1.5ex width.5em}}

\def\a{\alpha}
\def\endpf{\hbox{\vrule height1.5ex width.5em}}

\def\C{\hbox{C}}
\def\-{\overline}

\def\sm{\setminus}

\def\wt{\widetilde}
\def\ra{\rightarrow}
\def\p{\partial}

\def\endpf{\hbox{\vrule height1.5ex width.5em}}

\def\a{\alpha}

\def\C{\hbox{C}}
\def\a{\alpha}

\def\endpf{\hbox{\vrule height1.5ex width.5em}}
\def\a{\alpha}

\def\ra{\rightarrow}
\def\p{\partial}
\def\wt{\widetilde}
\def\-{\overline}

\newtheorem{theorem}{Theorem}[section]
\newtheorem{lemma}[theorem]{Lemma}

\newtheorem{proposition}[theorem]{Proposition}

\newtheorem{Definition}[theorem]{Definition}

\date{\ }

\begin{document}

\title{\bf Non-embeddability   into a fixed sphere for a family of  compact  real algebraic hypersurfaces }
\author{Xiaojun Huang\footnote{ Supported in part by
 NSF-1363418},\  \ \ \ Xiaoshan Li \  \ and \ \ Ming Xiao}

\maketitle

\section{Introduction}
Compact CR manifolds of hypersurface type play an important role in
the subject of Complex Analysis of Several Variables. For instance,
these manifolds include the small link of all isolated complex
singularities and, in particular, all exotic spheres of Milnor. In a
more geometric aspect, spheres are the model of strongly
pseudo-convex hypersurfaces. Motivated by various embedding theorems
in differential topology, Stein space theory, etc, it has been a
natural  question
 in Several Complex Variables to determine when a real hypersurface
$M\subset {\mathbb C}^n$  can be holomorphically embedded into the
sphere: $\mathbb{S}^{2N-1}:=\{\sum_{j=1}^N|z_{j}|^{2}=1\} \subset
\mathbb{C}^{N}$  for a sufficiently large $N$.

By a  holomorphic  embedding of $M\subset {\mathbb C}^n$ into $M'
\subset {\mathbb C}^N$, we mean a holomorphic embedding of an open
neighborhood $X$ of $M$ into a neighborhood $X'$ of $M'$, sending
$M$ into $M'$. It follows easily  that a hypersurface
holomorphically embeddable into a sphere ${\mathbb
S}^{2N-1}:=\{\sum_j |z_j|^2=1\}\subset\C^N$ is necessarily strongly
pseudoconvex and real-analytic. However, not every strongly
pseudoconvex real-analytic hypersurface can be  embedded into a
sphere of any dimension, as  shown by Forstneric [For] and Faran [Fa]
in the mid 1980s  based on  a Baire category argument. Explicit
examples of non-embeddable strongly pseudoconvex real-analytic
hypersurfaces were given much later by Zaitsev in [Zat] along with
explicit invariants serving as obstructions to the embeddability.

A recent  observation in [HZ] further shows that if a germ $M$ of a
strongly pseudoconvex algebraic hypersurface extends to a germ of
algebraic hypersurface with strongly pseudoconcave points or with
Levi non-degenerate points of positive signature, then $M$ can not
be holomorphically embedded into any sphere.

However, much less is  known about  the holomorphic embeddability of
an open piece of a  compact  strongly pseudoconvex hypersurface into
a sphere. In [HZ], using the local construction in [Zat], the
authors gave a compact real analytic strongly pseudoconvex
hypersurface, an open piece of which  can not embedded into  a
sphere. Also, in [HZ], it was shown that there are many compact
algebraic pseudoconvex hypersurfaces with just one weakly
pseudoconvex point any open piece of which can not be embedded into
any compact algebraic strongly pseudoconvex hypersurface which, in
particular, includes the spheres. For a related work on this, the reader
may also consult Ebenfelt and Son [ES]. Here, we should mention a
celebrated result of Fornaess [Forn] which states that any compact
smooth strongly pseudo-convex hypersurface in a complex Euclidean
space can be embedded into a compact strongly convex hypersurface in
${\mathbb C}^N$  for a sufficiently large $N$.  Though much
attention has been paid to the understanding of the embeddability
problem as discussed above, the following remains as  an open
question of long standing:

\medskip
{\bf Open Question}: {\it Is any compact strongly pseudoconvex real
algebraic hypersurface in $\mathbb{C}^n$ ($n\ge 2$) embeddable into
a sphere of a sufficiently large dimension?}

\medskip
Here recall that a smooth real hypersurface in an open subset $U$ of $\mathbb{C}^{n}$ is called real algebraic,
if it has a real-valued polynomial defining function.
%All real algebraic hypersurfaces in $\mathbb{C}^{n}$ are automatically smooth and closed.
%Related to this problem, it is also worth considering the following question:

%{\bf Question 2}: Is there an integer $N$ such that for any compact strongly pseudoconvex real-algebraic hypersurface in $\mathbb{C}^2$ can be holomorphically embeddable into $\partial\mathbb{B}^{N}$?

In this paper, we carry out a study along the lines of the above
open question. First, write

\begin{equation} \label{equhym}
M_\e=\{(z,w)\in\mathbb{C}^2:\rho=\varepsilon_0(|z|^8+c\mathrm{Re}|z|^2z^6)+|w|^2+|z|^{10}+\e
|z|^2-1=0\}.
\end{equation}
Here, $2<c<\frac{16}{7}$, $\varepsilon_0 >0$ is a sufficiently small
number  such that $M_\varepsilon$ is smooth for all $0\le
\e<1$. An easy computation shows that, for any $0 <
\e < 1,$ $M_{\e}$ is strongly pseudoconvex. Also,
it is easy to see that $M_\e$ is compact. $M_\e$ is a small
algebraic deformation of the famous Kohn-Nirenberg domain \cite{kn}.
Write $D_{\e}$ for the domain bounded by $M_\e.$
We prove the following result in this paper:

%We will give a negative answer to {\bf Question 2} by establishing
\begin{theorem}\label
{th1} For any positive integer $N$, there is a number $\e(N)$ with
$0<\e(N)<1$ such that for any $\e$ with $0<\e<\e(N)$,  the compact
algebraic strongly pseudoconvex hypersurface $M_\e$ can not
be locally holomorphically embedded into $\mathbb{S}^{2N-1}$.
Namely, for an open piece $U_\e$ of $M_\e$, any holomorphic
map sending $U_\e$ into $\mathbb{S}^{2N-1}$ must be a constant map.
\end{theorem}

Theorem \ref{th1} does not give yet a negative answer to the above
Open Question. However, it shows at least  that  the Whitney (or
Remmert ) type embedding theorem in differential topology (or in the
Stein space theory, respectively) does not hold  in the setting
considered in this Open Question. We notice that $M_\e$ can always
be embedded into a generalized sphere with one negative Levi
eigenvalue.  Indeed, this embedding  property  is a special case of
a general result of Webster [We] which concerns the holomorphic
embeddability of  an algebraic strongly pseudo-convex hypersurface
into a generalized sphere with one negative Levi eigenvalue. Since
the Segre families of generalized spheres with the same dimension are
biholomorphic to each other, we see that the Segre family of $M_\e$
can be holomorphically Segre-embedded into the Segre family of the
sphere in ${\mathbb C}^6$. We will explain this in more detail in
Remark 2.12.

\medskip
 Our proof is based on   the algebraicity theorem
in \cite{Hu} and  the work in Huang-Zaitsev [HZ], where it was shown
that $M_\e$ can not be embedded into any sphere when
$\e=0$. Unfortunately, the compact smooth algebraic hypersurface
$M_\e$ with $\e=0$ has Kohn-Nirenberg points [KN] which are
non-strongly pseudo-convex points. Our family of compact strongly
pseudoconvex hypersurfaces are small algebraic perturbation of the
Kohn-Nirenberg type domain $M_0$. Other main ideas in the paper
include the Segre variety technique developed in [HZ] to show the
rationality for a certain class of algebraic maps.

\section{Proof of Theorem \ref{th1}}

We divide the proof into many  small lemmas for clarity of the
exposition.

 We first fix needed notations. Let $M \subset U(\subset
\mathbb{C}^{n})$ be a closed real-analytic subset defined by a
family of real-valued real analytic functions
$\{\rho_{\alpha}(Z,\overline{Z})\},$ where $Z$ is the coordinates of
$\mathbb{C}^{n}.$ Assume that the complexification
$\rho_{\alpha}(Z,W)$ of $\rho_{\alpha}(Z,\overline{Z})$ is
holomorphic over $U \times conj(U)$ with
$$conj(U):=\{W:\overline{W} \in U\}$$
for each $\alpha.$ Then the complexification $\mathcal{M}$ of $M$ is
the complex-analytic subset in $U \times conj(U)$ defined by
$\rho_{\alpha}(Z,W)=0$ for each $\alpha.$  Then for $W \in
\mathbb{C}^{n},$ the Segre variety of $M$ associated with the point
$W$ is defined by $Q_{W}:=\{Z:(Z,\overline{W}) \in \mathcal{M}\}.$
In what  follows, we will write $\mathcal{M}_{\e}$ for the
complexification of $M_{\e}$ and write $\mathcal{M}'$ for
the complexification of $\partial \mathbb{B}^{N}.$ Similarly, we
will write $Q^{\e}_{p}$ for the Segre variety of
$M_{\e}$ associated with the point $p,$ and write $Q'_{q}$
for the Segre variety of $\partial \mathbb{B}^{N}$ associated with
the point $q.$ For any $p \in \mathbb{C}^{2},$  write
$p=(z_{p},w_{p})$ or $p=(\xi_{p},\eta_{p})$. The following lemma
proved in  [HZ] will be used in this paper:
\begin{lemma} \label{le21}
Let $U \subset \mathbb{C}^{n}$ be a simply connected open subset and
$\mathcal{S} \subset U$ be a closed complex analytic subset of
codimension one. Then for $p \in U \setminus \mathcal{S},$ the
fundamental group $\pi_{1}(U \setminus \mathcal{S},p)$ is generated
by loops obtained by concatenating (Jordan) paths
$\gamma_{1},\gamma_{2},\gamma_{3},$ where $\gamma_{1}$ connects $p$
with a point arbitrary close to a smooth point $q_{0} \in
\mathcal{S},$ $\gamma_{2}$ is a loop around $\mathcal{S}$ near
$q_{0}$ and $\gamma_{3}$ is $\gamma_{1}$ reversed.
\end{lemma}

Making use of the above lemma, we next prove the following  lemme:
(Notice that a local but a  general version of this result played an
important  role in the paper  [HZ].)
\begin{lemma} \label{222}
 Let $M_\e$ be defined as in (\ref{equhym}) with $p_0$ in
$M_\e$. Let $\mathcal{S}$ be a complex analytic
hyper-variety in $\mathbb{C}^2$ not containing $p_0$.
% and the complex hypersurface defined by $w=0$.
 Let
$\gamma\in\pi_1(\mathbb{C}^2\setminus \mathcal{S}, p_0)$ be obtained
by concatenation of $\gamma_1,\gamma_2,\gamma_3$ as described  in
Lemma \ref{le21}, where $\gamma_2$ is a small loop around
$\mathcal{S}$ near a smooth point $q_0\in \mathcal{S}$ with
$w_{q_0}\not =0$.  Then $\gamma$ can be slightly and homopotically
perturbed to a loop $\wt{\gamma} \in \pi_1(\mathbb{C}^2\setminus
\mathcal{S}, p_0)$ such that there exists a null-homotopic loop
$\lambda\in \pi_1(\mathbb{C}^2\setminus \mathcal{S}, p_0)$ with
$(\lambda, \overline{\wt\gamma})$ contained in the
complexification $\mathcal{M}_\e$ of $M_\e$. Also,
for an element $\hat{\gamma} \in \pi_{1}({\mathbb C}^2\setminus
\mathcal{S}, p_0)$ with a similar property described above,  after a
small perturbation to $\hat{\gamma}$ if needed, we can find a
null-homotopic loop in $\hat{\lambda} \in \pi_{1}({\mathbb C}^2
\setminus \mathcal{S}, p_0)$ such that
$(\hat{\gamma},\overline{\hat{\lambda}}) \subset
\mathcal{M}_{\e}.$
\end{lemma}
\begin{proof}
 First notice the fact that $Q_{p}^\e$ is  smooth when $w_{p} \neq
 0$ defined by $\eta=\varphi(\-{p},\xi)$ with $\xi\in {\mathbb C}^2$,
 where $\varphi$ is as in (\ref{111}) below:
\begin{equation}\label{111}
\varphi(\-{p},\xi)=\varphi(\overline z_p, \overline w_p,
\xi)=-\frac{\varepsilon_0(\xi^4\overline z_p^4+\frac
c2(\xi^7\overline z_p+\xi\overline z_p^7))+\xi^5\overline
 z_p^5+\e\xi\overline z_p-1}{\overline w_p},
\end{equation}

 Moreover, for any $q_{1} \neq q_{2}\in {\mathbb C}^2$ with $w_{q_{1}} \neq 0, w_{q_{2}} \neq 0$ and for any $U \subset \mathbb{C}^{2},$
 $ Q^{\e}_{q_{1}} \not\equiv Q^{\e}_{q_{2}}$ in $U$ unless they both are empty subset.
 After slightly perturbing   $p_0$ in $M_\e$, if needed, we can assume without loss of generality  that
 %$p_0 \not\in \mathcal{S}$ and
 $w_{p_0}\not = 0$.

 Now for any   $\xi\in {\mathbb C}$,
 % \in Q^{\varepsilon}_{p_{0}},$
%Pick $q^\prime=(\xi_{q^\prime}, \eta_{q^\prime})\not \in {\mathcal S}$.
we define a map ${\cal R}_{\xi}(z,w)=(\xi, \varphi(\overline z, \overline
w, \xi))$ from $\mathbb{C}^{2}\sm \{w\not = 0\}$ into
$\mathbb{C}^{2},$
% where
%\begin{equation}\label{111}
%\varphi(\overline z, \overline w,
%\xi_{q^\prime})=-\frac{\varepsilon_0(\xi_{q^\prime}^4\overline
%z^4+\frac c2(\xi_{q^\prime}^7\overline z+\xi_{q^\prime}\overline
%z^7))+\xi_{q^\prime}^5\overline
% z^5+\varepsilon\xi_{q^\prime}\overline z-1}{\overline w},
%\end{equation}
  which is
anti-holomorphic in $(z,w)$ for $w\neq 0$ and is real analytic in
all variables away from $w=0$. Also, if we write
$p_0=(\xi_{p_0},\eta_{p_0})$, then $(\xi_{p_0}, \varphi(\-{p_0},
\xi_{p_0}))=p_0$ and thus ${\cal R}_{\xi_{p_0}}(p_0)=p_0$. From the
defintion, we see that ${\cal R}_{\xi}$ sends $(z,w)$ to
$Q_{(z,w)}^{\e}.$

We claim that, possibly away  from a certain nowhere dense closed
subset in ${\mathbb C}$ for $\xi$,  for a generic smooth point $q$
in the irreducible branch of $ {\mathcal S}$ containing $q_0$ as in
the lemma, there is a sufficiently small ball $\Omega_q$ centered at
$q$ (whose size may depend on $q$) such that ${\cal R}_{\xi}$ maps
$\Omega_q$ into a small open ball $B_q$ with $B_q\cap
\mathcal{S}=\emptyset$. Suppose not. Then we have a smooth piece $E$
from the branch described above of ${\mathcal S}$  such that
${\cal R}_{\xi}(E)$ is contained in ${\mathcal S}$ for any $\xi$ in a
certain open subset first and then for all $\xi $  by the uniqueness
of analytic functions. Letting $\xi=0$, we see that the branch
containing all these images must be defined by $z=0$ unless $E$ is
defined by $w=constant$. However, if the branch containing $E$ is
defined by $w=constant$, by making $\xi\not = 0$, we easily see that
the union of ${\cal R}_{\xi}(E)$ as $\xi$ varies occupies an open subset of
${\mathbb C}^2$. This is a contradiction again.

Now, we fix a $\xi_{0}$ as in the above claim and also assume without
loss of generality that $\xi_{0}$ is the first coordinate $\xi_{p_{0}}$ of $p_0$ (for
we are certainly always allowed to perturb $p_0$ inside $M_\e$ to
achieve this). Back to our loop $\gamma$, we now deform
 $\gamma_1, \gamma_2,
\gamma_3$  to $\wt\gamma_1, \wt\gamma_2, \wt\gamma_3$ respectively.
Here $\wt \gamma_1$ connects $p_0$ with a point  $q^*$ in a small
ball $\Omega$ centered at a certain smooth point $q\in {\mathcal
S}\approx q^{*}$, $\wt\gamma_2$ is a loop based at $q^*$ around
$\mathcal{S}$ inside $\Omega$ and sufficiently close to $q$, and
$\wt\gamma_3$ is $\wt\gamma_1$ reserved such that the loop
$\wt\gamma$ obtained by concatenation of $\wt\gamma_1,
\wt\gamma_2,\wt\gamma_3$ is the same as $\gamma$ as elements
in $\pi_1(\mathbb{C}^2\setminus \mathcal{S}, p_0)$. Moreover,
${\cal R}_{\xi}(\Omega)$ is contained in a ball not cutting $\mathcal S$.
Also, we  assume that the $w$-coordinate of points in $\wt\gamma(t)$
never vanishes.
 %is homotopic to $\gamma$ in
%$\pi_1(\mathbb{C}^2\setminus \mathcal{S}, p_0),$ and such that if we
% Write $\gamma(t)=(z(t),w(t)),~0 \leq t \leq 1.$
 % it holds that $w(t) \neq 0$ for all $t.$
 Now define $\lambda_{2}={\cal R}_{\xi_{0}}(\wt\gamma_2)$.
%By the property of the map
%$R_{\xi_{q^\prime}}$, $\lambda_2$ is contained in
%$B(q^\prime,\delta)$.
We  choose a suitable path $\{\xi(t):0 \leq t \leq 1\}$ in $\mathbb{C}$ with $\xi(0)=\xi(1)=\xi_{0}$ such that if  we
define $\lambda_1={\cal R}_{\xi(t)}(\wt\gamma_1),$ then $\lambda_{1}$
avoids $\mathcal{S}$(with possibly a slight perturbation of $\wt\gamma_{1}$ fixing endpoints). Furthermore, if we define $\lambda_3$ to be
the reverse of $\lambda_1,$ and $\lambda$ to be the concatenation of
$\lambda_1, \lambda_2, \lambda_3,$ then $\lambda$ is a
null-homotopic loop in $\pi_1(\mathbb{C}^2\setminus \mathcal{S},
p_0).$ Moreover, $(\lambda(t), \overline{\wt\gamma(t)})$ is in
the complexification $\mathcal{M}_{\e}$ of
$M_{\e}$ by the way it was constructed. The last statement
in the lemma follows from the symmetric property of Segre variety
and what we just proved.
\end{proof}

\begin{proposition}\label{lem4}
 For an $\e$ with  $ 0<\e<1$,  assume that $F$ is non-constant holomorphic
 map from an open piece of
 $M_\e$ into $\partial\mathbb{B}^N$ ($N \in \mathbb{N}).$
Then $F$ extends to  a proper rational  map from $D_\e$ into
${\mathbb B}^N$, holomorphic over $\-{D_\e}$.
\end{proposition}

\begin{proof}

%By Hartogs' extension theorem, $F$ can be extended to a neighborhood of $\overline{D_{\varepsilon}}.$
 By a theorem of the first author in  [Hu], $F$ is complex algebraic (possibly multi-valued).  In particular, any branch of $F$ can be holomorphically
  continued along a path not cutting a certain proper complex algebraic subset $\mathcal{S} \subset \mathbb{C}^{2}.$
  We  need only to prove the proposition assuming that $\mathcal{S}$ is a hyper-complex analytic variety.
 Seeking a contradiction, suppose not. Then we can
  find a point $p_{0} \in U\subset  M_{\e},p_{0}=(z_{0},w_{0})$ with $w_{0} \neq 0,$ a loop
   $\gamma \in \pi_{1}(\mathbb{C}^{2} \setminus \mathcal{S},p_{0})$ obtained by concatenation
    of $\gamma_{1},\gamma_{2},\gamma_{3}$ as in Lemma 2.1, where $\gamma_{2}$
    is a small loop around $\mathcal{S}$ near a smooth point $q_{0} \in \mathcal{S},$
     such that when we holomorphically continue $F$ from a neighborhood of $p_{0}$ along $\gamma$ one round,
      we will obtain another branch $F_{2}(\neq F)$ of $F$ near $p_{0}.$ Obviously, we can assume $q_{0}$ is a
      smooth point of some branching hypervariety $\mathcal{S}' \subset \mathcal{S}$ of $F.$ We next proceed  in two steps:

\medskip
{\bf{Case I:}}  If we can find a loop $\gamma$ as above such that the corresponding  $\mathcal{S}' \neq  \{w=0\},$ by perturbing $\gamma$ if
 necessary, we can make $w_{q_{0}} \neq 0.$  By Lemma \ref{222}, after slightly perturbing $\gamma$ if necessary, there exists a null-homotopic loop
  $\lambda$ in $\pi_{1}(\mathbb{C}^{2} \setminus \mathcal{S},p_{0})$ with $(\gamma,\overline{\lambda})$ contained in the complexification
  $\mathcal{M}_{\e}$ of $M_{\e}$ We know that $(F,\overline{F}):=(F(\cdot),\overline{F(\overline{\cdot})})$ sends a neighborhood
   of $(p_{0},\overline{p_{0}})$ in $\mathcal{M}_{\e}$ into $\mathcal{M}'.$ Applying the analytic continuation along the loop
    $(\gamma,\overline{\lambda})$ in $\mathcal{M}_{\e}$ for $\rho(F,\overline{F}),$ one concludes by the uniqueness of analytic
    functions
     that $(F_{2},\overline{F})$ also sends a neighborhood of $(p_{0},\overline{p_{0}})$ in $\mathcal{M}_{\e}$ into $\mathcal{M}'.$
     Consequently, we get $F_{2}(Q_{p}) \subset Q'_{F(p)}$ for  $p \in M_{\e}$ near $p_{0}.$ In particular,
     %$F_{2}$ maps a neighborhood of $p_{0}$ in $M_{\varepsilon}$ into $M'.$
     we have the following:
\begin{equation}\label{eqdiffb}
F_{2}(p) \in Q'_{F(p)}, \forall p \in M_{\e},~ p\approx p_{0}.
\end{equation}
Now applying the holomorphic continuation along the loop
$(\lambda,\overline{\gamma})$ in $\mathcal{M}_{\e}$ for
$\rho(F_{2},F),$ we get by uniqueness of analytic functions that
$(F_{2},\overline{F_{2}})$ sends a neighborhood of $(p_{0},\-{p_0})$
in $\mathcal{M}_{\e}$ into $\mathcal{M}'.$ Hence, we also
have
\begin{equation}\label{eqdiffb1}
F_{2}(p) \in Q'_{F_2(p)}, \forall p \in M_{\e},~ p\approx
p_{0}.
\end{equation}
In particular, $F_{2}(p)\in \p{\mathbb B}^N$.
% maps a neighborhood of
%$p_{0}$ in $\mathcal{M}_{\varepsilon}$ into $\mathcal{M}'.$
Combining this with equation (\ref{eqdiffb}), and noting that for
any $q \in
\partial \mathbb{B}^{N}, \partial \mathbb{B}^{N} \cap Q'_{q}={q},$
we get $F_{2}(p)= F(p)$ for any $p \in M_{\e}$ near
$p_{0}.$ Thus $F_{2} \equiv F$ in a neighborhood of $p_{0}$ in
$\mathbb{C}^{2},$ which is a contradiction.

\medskip {\bf{Case II:}} \ Now, suppose  $ W:=\{w=0\}$ is the only branching locus of the algebraic extension of
$F$. Since $W$ is smooth and $\pi_1({\mathbb C}^2\setminus
W)={\mathbb Z}$, we  get the cyclic branching property for $F$. Now,
we notice that $W$ cuts $M_\e$ transversally at a certain point
$p^*=:(z_0,0)$.
%Hence $\widetilde{F}(z,w)=F(z,w^{k})$ is
%single-valued over ${\mathbb C}^2$ for a certain positive integer
%$k$.
When we will continue along loops inside $T^{(1,0)}_{p^*}M_\e$ near
$p^*$, we recover all branches of $F(z,w)$. Since any loop inside
$T^{(1,0)}_{p^*}M_\e$ near $p^*$ can be easily homotopically
deformed into loops in $M_\e$ near $p^*$, we conclude that
% when we holomorphically deform a  loop inside $M_\e$ near $p^*$,
we recover
all branches of $F$ near $p^*$ by continuing any branch of $F$ near
$p^*$ along loops inside $M_\e\sm W$ near $p^*$. Hence, we are now
reduced to the local situation as encountered in Proposition 3.10 of
[HZ]. Hence, by Proposition 3.10 of [HZ], for  $Z(\not =)\approx p^*$
and two barnches $F_1$ and $F_2$ of $F$ near $Z$, we have $F_1(Z),
F_2(Z)\in Q'_{F_1(Z)}\cap Q'_{F_2(Z)}.$ As above, we see that
$F_1(Z)=F_2(Z)$.
% By the analyticity, any branch
%of $F$ must also map $M_\e$ into the sphere. We conclude that all
%branches of $F(Z)$ approach to the same limit as $Z\in M_\e$
%approaches to $p^*$. By Remark 3,1 (a) of [HZ],
We thus conclude that $F$ is single-valued.
% irreducible component of $\mathcal{S},$ and
%For any loop $\gamma$ described as above, the corresponding
%branching hypervariety $\mathcal{S}_{0}$ always equals to $\{w=0\}.$

 Since $F$ is algebraic, it is rational.
Once we know that $F$ is a rational map from $M_ \e $ into the
sphere, by a theorem of Chiappari [Ch], we know that $F$ extends to
a holomorphic map from a neighborhood of $\-{D_\e}$ and properly
maps $ {D_\e}$ into the ball. This completes the proof of
 Proposition \ref{lem4}.

\end{proof}
\bigskip

 Next we  recall the following
definition.

\begin{Definition}
Let $F$ be a rational map from $\mathbb{C}^{n}$ into
$\mathbb{C}^{m}.$ We write
$$F=\frac{(P_{1},\cdots,P_{m})}{R}$$
where $P_{j},j=1,\cdots,m$ and $R$ are holomorphic polynomials and
$(P_{1},\cdots,P_{m},R)=1.$ The degree of $F,$ denoted by
$\mathrm{deg}F,$ is defined to be
$$\mathrm{deg}F: =max\{\mathrm{deg}(P_{j}),j=1,\cdots,m,\mathrm{deg}R\}.$$
\end{Definition}

To emphasize on the dependence on the  parameter $\e,$ in
what   follows, we write $F^\e$ for a holomorphic map
 from a certain open piece of $M_\e$ into $\partial \mathbb{B}^{N}.$ By what we did above, $F^\e$ extends  to holomorphic map over a neighborhood
 of $\-{D_\e}$. The purpose
of the next three lemmas is to show the uniform boundedness of the
degree of $F^{\e}.$ We  mention a related article of Meylan
in [Mey] for the uniform estimate of degree for proper rational maps
between balls.

\begin{lemma}\label{lem5}
Let $F^{\e}$ be a proper rational map from  $D_\e$ into
${\mathbb B}^N$ holomorphic over $\-{D_\e}$. Then there is an open
piece $U$ of $M_\e$ such that for any $p\in U$ with $w_p\not =0$ and
we have $\mathrm{deg}F^\e |_{Q_{p}}\leq d,$ where
$d=\frac{7N(N+1)}{2}.$ Here we set $F^\e |_{Q_{p}}:=
F^\e(\xi, \phi(\-{z_p},\-{w_p},\xi))$ with
$\phi(\-{z_p},\-{w_p},\xi)$ as in (\ref{111}), which is a
holomorphic  polynomial function in $\xi$.

\end{lemma}

\begin{proof}
Let $p_0=(z_0, w_0)\in M_\e$ with $w_{p_0}\not =0$.
% After a rotation if needed, we can assume that $F(p)=(0,1)$.
For any $(\xi,\eta)\in Q_{p_0}$, we have
\begin{equation}\label{eqn23}
F_1^\e(z,w)\overline{F_1^\e(\xi,
\eta)}+\cdots+F_{N}^\e(z,w)\overline{F_{N}^\e(\xi,\eta)}=1,\
(z,w)\in Q_{(\xi,\eta)}.
\end{equation}
%with $Q_{(\xi,\eta)}$ a smooth complex submanifold.

 Here we write $F^\e=(F_1^\e, \cdots,
F_{N}^\e).$ Recall $Q_{(\xi,\eta)}$ is given by
$\varepsilon_0(z^4\overline\xi^4+\frac c2(\overline\xi
z^7+z\overline\xi^7))+w\overline\eta+z^5\overline\xi^5+\e
z\overline\xi-1=0.$ Write

\begin{equation}
\mathcal{L}=(4\varepsilon_0\overline\xi^4z^3+\frac{7c\varepsilon_0}{2}\overline\xi
z^6+\frac{c\varepsilon_0}{2}\overline\xi^7+5\overline\xi^5z^4+\e\overline\xi)\frac{\partial}{\partial
w}-\overline\eta\frac{\partial}{\partial z}.
\end{equation}
Then $\mathcal{L}$ forms a basis for the holomorphic  tangent vector
fields of $Q_{(\xi,\eta)}$ near $(z, w)\in Q_{(\xi, \eta)}$. When
$(\xi,\eta)=(z,w)$ and moves along $U \subset M_{\e}$, ${\mathcal L}$ reduces to
the CR vector field along $U\subset M_\e$. Applying
$\mathcal{L}^\alpha,\ |\a|>0$, to (\ref{eqn23}) and evaluating at
$p_0$, one gets
\begin{equation}\label{eqn24}
\mathcal{L}^\alpha F_1^\e(z_0,
w_0)\overline{F_1^\e(\xi,\eta)}+\cdots+\mathcal{L}^\alpha
F_{N}^\e(z_0,w_0) \overline{F_{N}^\e(\xi,\eta)}=0,
\ |\a|> 0.
\end{equation}
Write
\begin{equation}
V_\alpha^\e(\xi, \eta)=(\mathcal{L}^\alpha
F_1^\e(z_0,w_0),\cdots,\mathcal{L}^\alpha
F_{N}^\e(z_0, w_0)).
\end{equation}
Choose $U\subset M_\e$ such that $\{V_\alpha^\e(z_0,
w_0)\}_{\a>0}^{\infty}$ has a constant rank $k\le N$ for
$(z_0,w_0)\in U$. Then, after shrinking $U$ if needed, by a calculus
computation (see [La], for instance) we conclude that
$\{V_\alpha^\e(z_0, w_0)\}_{\a>0}^{k}$ must be a basis of
$\{V_\alpha^\e(z_0, w_0)\}_{\a>0}^{\infty}$. Making use of
the Taylor expansion, we see that the linear span of
$\{V_\alpha^\e(z_0, w_0)\}_{\a>0}^{k}$ is the smallest
subspace containing $F^{\e}(Q_{(z_0,w_0)})-F^{\e}(z_0,w_0)$.

\begin{itemize}
  \item If $k=N-1$ in $U$,
   we can solve for $F^\e(\xi,\eta)$ for $(\xi,\eta)\in
   Q_{(z_0,w_0)}$
  from Equation (\ref{eqn23}) and  (\ref{eqn24})  by the Cramer rule. Notice that $\eta=\phi(\-{p_0},\xi)$ is solved as a polynomial function of
  $\xi$ of degree $7$.
  Therefore, as a rational function in $\xi$,  we get $$\mathrm{deg} F^\e|_{Q_{(z_0,w_0)}}\leq
  d$$ for $(z_0,w_0)\in U$.
  \item If $k<N-1 $, then one can find constant vectors ${\bf{V}}_1, \cdots, {\bf{V}}_{N-k}$ in $\mathbb{C}^{N}$ such that
   $$\mathrm{Span}\{{\bf{V}}_1, \cdots, {\bf{V}}_{N-k}\}\bigoplus\mathrm{Span}\{V_\alpha^\e(z_0, w_0) \}_{1\leq\alpha\leq k}=\mathbb{C}^{N-1}$$ and
    ${\bf{V}}_i\cdot (\-{F^\e(\xi,\eta)-F^{\e}(z_0,w_0)})=0$ on $Q_{(z_0,w_0)}, 1\leq i\leq N-k$.
    One can still apply Cramer's rule to solve for $F^\e(\xi,\eta)$
    with $(\xi,\eta)\in Q_{(z_0,w_0)}$ to show, as a rational
    function of $\xi$, that
      \begin{equation}
      \mathrm{deg} F^\e|_{Q_{(z_0,w_0)}}<d.
            \end{equation}
\end{itemize}
This completes the proof of the lemma.
\end{proof}

{\bf Remark}: The above argument  can be use to  show directly that
$F$ is rational (as a function in $\xi$) when restricted to a Segre
variety. However this type of information is not enough, in general,
to conclude the rationality of $F$: Let $M \subset \mathbb{C}^2$ be a strongly pseudoconvex
hypersurface defined by
$|w|^2=(1+|z|^2)^2$ and $g=\sqrt{w}$. The Segre variety $Q_{(z,w)}$
of $M$ for each $(z,w)$ is defined by $w\-{\eta}=(1+z\-{\xi})^2$.
$g|_{Q_{(z,w)}}=\pm\frac{1+\-{z}{\xi}}{\-{\sqrt{{w}}}}$, which is a
polynomial as a function in $\xi$ for $w\not = 0$.

%\bigskip
 %The following lemma is inspired by  Lemma 5.4 in [HJ]:
\begin{lemma}\label{lem6}
Let $H=\frac{(P_{1},\cdots,P_{N})}{R}$ with $R(0,0)\neq 0$ be a
rational map from $\mathbb{C}^2\sm \{R=0\}$ into $\mathbb{C}^{N}$,
where $P_{j}, j=1,\cdots,N,R$ are holomorphic polynomials and
$(P_{1},\cdots,P_{N},R)=1.$ Assume that there is an open subset $U$
of $M_\e$ such that for each $p\in U$ with $w_p\not =0$ and, as a
rational function in $\xi$, $\mathrm{deg}(H|_{Q_p})\leq k$ with
$k>0$ a fixed integer. Then $\mathrm{deg}(H)\leq k$.
\end{lemma}
\begin{proof}
Set
\begin{equation}
A=\{(\xi,\eta)\in\mathbb{C}^2: P_{1}(\xi,\eta)=\cdots= P_{N}(\xi,\eta)=R(\xi,\eta)=0\}.
\end{equation}
Then $A$ has at most   finitely many points. It is easy to see that
if $Q_p$ does not pass through any point of $A$, then as a rational
function in $\xi$, the degree of $H|_{Q_p}$ is the same as the
degree of $H$ as a rational function in all variables.
%\begin{equation}
%\mathrm{deg}H|_{Q_p} = \mathrm{deg}H.
%\end{equation}
Thus it only remains to show the existence of $(z_0,w_0)\in U$ such
that $Q_{(z_0,w_0)}\cap A=\emptyset$. Indeed, fix $(\xi_0,\eta_0)\in
A$, then $\xi_0\neq0$ or $\eta_0\neq0$. $(\xi_0,\eta_0)\in
Q_{(z_0,w_0)}$ if and only if
\begin{equation}
\varepsilon_0\overline\xi_0^4
z_0^4+\frac{c}{2}\varepsilon_0(\overline\xi_0 z_0^7+z\overline
\xi_0^7)+ w_0\overline \eta_0+ z_0^5\overline\xi_0^5+
\e z_0\-\xi_0=1.
\end{equation}
The collection of such pairs $\{(z_0,w_0)\}$ is a complex subvariety
of complex dimension $1.$ Thus $\{(z,w)\in\mathbb{C}^2:
Q_{(z,w)}\cap A \neq \emptyset\}$ is a finite union of   complex
 subvarieties of complex dimension $1$. But $U\subset M_{\e}$ is of real dimension $3.$
 Thus there exists $(z_0, w_0)\in U$ such that $Q_{(z_0,w_0)}\cap A=\emptyset$.
\end{proof}

Notice that our $F^\e$ is holomorphic in $D_\e$ and thus at $0$. As
a consequence of Lemma \ref{lem5} and Lemma \ref{lem6}, we have the
following:

\begin{lemma}\label{lem7}
Let $F^\e, d$ be as in Lemma \ref{lem5}. Then $\mathrm{deg} F^\e\leq d.$
\end{lemma}

The following three lemmas are to show the uniform boundedness of
the coefficients of $F^{\e}$.

\begin{lemma}\label{lem8}
Let $p(z)=\sum\limits_{i=1}^m a_iz^i+1$ be a holomorphic polynomial in $\mathbb{C}$. Assume that $p(z)\neq0$ in $\Delta$, where $\Delta$ is the unit disk centered at $0$ in $\mathbb{C}$. Then $|a_i|\leq C_m$ for all $1\leq i\leq m$, where $C_m$ is a constant depending only on $m$. Consequently, $|p(z)| \leq mC_{m}+1$ in $\Delta.$
\end{lemma}
\begin{proof}
We write $p(z)=a_k\Pi_{i=1}^k(z-z_i)$, where $1\leq k\leq m$ is the
largest number $l$ such that $a_l\neq 0$, and $\{z_i\}_{i=1}^k$ are
the roots of $p(z)$ in $\mathbb{C}$. Notice that $p(0)=1$ and
$p(z)\neq 0$ in $\Delta$, we get $|z_i|\geq1$ for all $1\leq i\leq
k$, and $|a_k\Pi_{i=1}^k z_{i}|=1$. Thus $|a_k|\le 1$. Moreover,  by
applying Vieta's formula, we have for each $1 \leq j\le k-1$,
$$|a_{k-j}|=|\frac{\sum_{l_1<\cdots <l_j}z_{l_1}\cdots z_{l_j}}{\Pi_{i=1}^k
z_{i}}|\le C_m$$ for a certain constant $C_m$ depending only on $m$.
\end{proof}
\begin{lemma}\label{lem9}
Let $p(z)=\sum_{|\alpha|=1}^m a_{\alpha}z^\alpha+1$ be a holomorphic polynomial in $\mathbb{C}^N, N\geq 1$. Assume that $p(z)\neq 0$ in $\mathbb{B}^N$. Then $|a_\alpha|\leq \widetilde{C}_{m}$ for all $1\leq|\alpha|\leq m$, where $\widetilde{C}_{m}$ is a positive constant depending only on $m$.
\end{lemma}
\begin{proof}
 Fix $z\in\p\mathbb{B}^N.$ Set $\tilde p(\xi)=p(\xi z), \xi\in \Delta,$ which is a holomorphic polynomial in $\mathbb{C}.$
 Noting that $\tilde{p}(\xi) \neq 0$ in $\Delta,$ by Lemma \ref{lem8},  $|\tilde p(\xi)|\leq mC_m+1,$ where $C_{m}$ is as in Lemma \ref{lem8}.
  Consequently, $|p(z)|\leq mC_m+1, \forall z\in \mathbb{B}^N$. By the Cauchy estimate, we conclude that there exists some constant $\widetilde{C}_{m}$ such that $|a_\alpha|\leq  \widetilde{C}_{m}$ for all $1\leq |\alpha|\leq m$.
\end{proof}
\begin{lemma}\label{lem10}
Let $F^\e, d$ be as in Lemma \ref{lem5} and assume that
$F^\e(0)=0$. Write
$F^\e(z,w)=\frac{P^\e(z,w)}{Q^\e(z,w)}$,
where $P^\e(z,w)=\sum\limits_{1\leq i+j\leq
d}a_{ij}^\e z^iw^j, Q^\e(z,w)=\sum\limits_{1\leq
i+j\leq d}b_{ij}^\e z^iw^j+1$. Moreover $(P^\e,
Q^\e)=1$. Then $|a_{ij}^\e|\leq C,
|b_{ij}^\e|\leq C$ for some constant $C$ depending only on
$N$.
\end{lemma}
\begin{proof}

Notice that there exists $r>0$ independent of $0<\e<1$ such
that $B(0,r)\subset D_\e$ and $ Q^\e(z,w)\neq 0$
in $B(0,r).$  As an application of Lemma \ref{lem9}, one can show
the uniform boundedness of $|b_{ij}^\e|$ by considering
$\tilde Q^\e (z,w)=Q^\e(\sqrt rz,\sqrt rw)$.
Consequently, $P^{\e}$ is uniformly bounded in $B(0,r)$ for
all $\e.$ And the uniform boundedness of
$a_{ij}^\e$ follows from the Cauchy estimate.
\end{proof}
Set $M_{0}=\{(z,w)\in\mathbb{C}^2:\rho=\varepsilon_0(|z|^8+c\mathrm{Re}|z|^2z^6)+|w|^2+|z|^{10}-1=0\}.$ Notice that $M_{0}$ has the Kohn-Nirenberg property at the point $(0,1).$ Here recall that (see [HZ]) a real hypersurface $M \subset \mathbb{C}^{n}$ is said to satisfy the Kohn-Nirenberg property at $p \in M,$ if for any holomorphic function $h \not \equiv 0$ in any neighborhood $U$ of $p$ in $\mathbb{C}^n$ with $h(p)=0,$ the zero set $\mathcal{Z}$ of $h$ intersects $M$ transversally at some smooth point of $\mathcal{Z}$ near $p.$ As an immediate application of Theorem 3.6 in [HZ], one has the following lemma,
\begin{lemma} \label{lem11}
Let $M_{0}$ be as above. Then any holomorphic map sending  an open piece of $M_{0}$ into $\partial \mathbb{B}^N$ is a constant.
\end{lemma}

\bigskip
We are now ready to prove our main theorem.

\medskip
 {\bf Proof of Theorem
\ref{th1}}. Seeking a contradiction, suppose the statement in the
main theorem does not hold. Then  for a certain positive integer $N$
and for a certain sequence $1>\e_k\ra 0^+$,
$M_{\e_k}$ are locally holomorphically embeddable into
$\mathbb{S}^{2N-1}$ for any $\e_k$. For each of such
$\e_k$, write a local holomorphic  embedding as
$F^{\e_k}$. Then, by Lemma \ref{lem4}, $F^{\e_k}$
extends to a rational and holomorphic map over $\-{D_\e}$.
%M_\varepsilon\rightarrow\partial \mathbb{B}^{N}$ for any
%$0<\varepsilon<1$.
 After composing with an automorphism of
$\mathbb{B}^{N}$, we can assume that $F^{\e_k}(0)=0$.
% Then by Lemma \ref{lem4}, $F_\varepsilon$ is
%rational  and is  holomorphic over $D_\e$.

 By Lemma \ref{lem7} and Lemma \ref{lem10}, we can  write
\begin{equation}
F^{\e_k}(z,w)=\frac{\sum\limits_{i+j=1}^d
a_{ij}^{\e_k} z^iw^j}{\sum\limits_{i+j=1}^d
b_{ij}^{\e_k} z^iw^j+1},
\end{equation}
where $d=\frac{7N(N+1)}{2}$ and $|a_{ij}^{\e_k}|\leq C,
|b_{ij}^{\e_k}|\leq C$ for all $i,j$ with
 $C$  a constant as in Lemma \ref{lem10}.
 Hence after passing to a subsequence if necessary, we can assume
 that
  %there exists a sequence $\{\varepsilon_k\}_{k=1}^\infty$ with $0<\varepsilon_k<1$ for all
 %% $k$ and $\varepsilon_k\rightarrow0$ as $k\rightarrow\infty$ such that
  $a_{ij}^{\e_k}\rightarrow a_{ij}, b_{ij}^{\e_k}\rightarrow b_{ij}$
  as $k\rightarrow\infty$ for some $a_{ij}\in\mathbb{C}, b_{ij}\in\mathbb{C}$ for all $i,j$.
  Set $F(z,w)=\frac{P(z,w)}{Q(z,w)}$, where $P(z,w)=\sum\limits_{i+j=1}^d a_{ij}z^iw^j$ and $ Q(z,w)=\sum_{i+j=1}^d
   b_{ij}z^iw^j+1$.
  Let $V=\{(z,w)\in\mathbb{C}^2: Q(z,w)=0\}$ be the variety defined by the zeros of $ Q(z,w)$
in $\mathbb{C}^2$. It is easy to see that for any open subset
$K\subset\subset \mathbb{C}^2\setminus V$, we have
$F^{\e_k}$ converges to $F$ uniformly in $K$. Pick $p_0\in
M\setminus V$ and a neighborhood $U$ of $p_0$ with
$U\subset\subset\mathbb{C}^2\setminus V$. $F^{\e_k}$
converges to $F$ uniformly in $\overline U$. Notice that for any
$p\in U\cap M$, there exists $p_k\in M_{\e_k}$ such that
$p_k\rightarrow p$ as $k\rightarrow\infty$. Then
$\|F(p)\|=\lim_{k\rightarrow\infty}\|F^{\e_k}(p_k)\|=1$. By
Lemma \ref{lem11}, $F$ is a constant map from
 $D_\e\cap M$ into the sphere. This   is a contradiction, for we know that $F(0)=0$. The proof of Theorem
\ref{th1} is complete.
\endpf

\bigskip
{\bf Remark 2.12}. It is clear that with the same proof, we can
construct a lot of more similar examples as in Theorem \ref{th1}.

Next, to see that $M_\e$ can be holomorphically embedded into the
generalized sphere in ${\mathbb C}^6$ with one negative Levi
eigenvalue, we observe that $
\hbox{Re}(|z|^2z^6)=\frac{1}{4}(|z^7+z|^2-|z^7-z|^2).$  Thus the map
$$F(z,w)=(\sqrt{\varepsilon_0}z^4, \frac{1}{2}\sqrt{\varepsilon_0 c}(z^7+z), w, z^5,
\sqrt{\e}z,  \frac{1}{2}\sqrt{\varepsilon_0 c}(z^7-z))$$
 holomorphically embeds $M_\e$ into the generalized sphere in ${\mathbb C}^6$ defined
 by ${\mathbb S}^{11}=\{(Z_1,\cdots,Z_6)\in {\mathbb C}^6: \
 \sum_{j=1}^{5}|Z_j|^2-|Z_6|^2=1\}.$

\noindent X. Huang, Department of Mathematics, Rutgers University at
New Brunswick, New Jersey 08903, USA. (huangx$@$math.rutgers.edu)

\noindent X. Li, School of Mathematics and Statistics, Wuhan
University, Hubei, Wuhan 430072, China. (xiaoshanli$@$whu.edu.cn)

\noindent M. Xiao, Department of Mathematics, Rutgers University at
New Brunswick, New Jersey 08903, USA. (mingxiao$@$math.rutgers.edu)

\begin{thebibliography}{99999}

\bibitem [Ch] {ch} S. Chiappari, {\it Holomorphic extension of proper
meromorphic mappings}, {\it Mich. Math. J.} 38, 167- 174 (1991).

\bibitem [ES]{ES} P. Ebenfelt and D. Son, {\it On the existence of holomorphic embeddings of strictly pseudoconvex algebraic hypersurfaces into spheres},
 May, 2012. (arXiv:1205.1237).


\bibitem [Fa]{Fa} J. J. V. Faran, {\it The nonimbeddability of real hypersurface in spheres.} {\it Proc. Amer. Math. Soc.}103, 3(1988), 902-904.

\bibitem [Forn]{forn} J. E. Forn{\ae}ss, {\it Strictly pseudoconvex domains in convex
domains}, {\it Amer. J. Math.} 98 (1976), 529-569.

\bibitem[For]{For} F. Forstneric, {\it Embedding strictly pseudoconvex domains into balls}, {\it Transations of AMS} (Vol.295), No.1 (May, 1986), 347-368.

\bibitem [HJ] {HJ} X. Huang and S. Ji, {\it Mapping $\mathbb{B}^{n}$ into $\mathbb{B}^{2n-1}$}, {\it Invent.Math.}, 145(2001), 219-250.
\bibitem [Hu] {Hu} X.Huang, {\it On the mapping problem for algebraic real hypersurfaces  in the complex spaces of different dimensions},
 {\it Annales de l'institut Fourier}, 44.2(1994) 433-463.
\bibitem[HZ]{HZ} X. Huang, D. Zaitsev, {\it Non-embeddable real algebraic hypersurface}, {\it Math. Z.}, 275, No. 3-4, 657-671 (2013).

\bibitem [KN]{kn}
J. J. Kohn and L. Nirenberg, {\it A pseudo-convex domain not admitting a
holomorphic support function}, {\it Math. Ann.}, 201 (1973), 265-268.

\bibitem [La] {La} B. Lamel, {\it A reflection principle for real-analytic
submanifolds of complex spaces}, {\it J. Geom. Anal.} 11 , no. 4,
625-631, (2001).

\bibitem [Mey] {mey} F. Meylan, {\it Degree of a holomorphic map between unit balls from $\mathbb
C^{2}$ to $\mathbb C^{n}$},
   {\it Proc. Amer. Math. Soc.} 134 (2006), 1023-1030.

\bibitem [We] {We1} S. M. Webster, {\it Some birational invariants for algebraic real
hypersurfaces},  {\it Duke Math. J.}, Volume 45, Number 1 (1978), 39-46.


\bibitem [Zat] {Z} D. Zaitsev, {\it Obstructions to embeddability into hyperquadrics and explicit examples},
{\it Math. Ann.}, 342 (2008), no. 3, 695-726.

\end{thebibliography}
\end{document}